%% file: main.tex
\documentclass[a4paper,12pt,leqno]{amsart}
\input{Preamble}
\usepackage{verbatim}

\begin{document}

\title[]{On the virtually cyclic dimension of\\ normally poly-free groups}

\author[Rita Jiménez Rolland]{Rita Jiménez Rolland}

\address{Instituto de Matemáticas, Universidad Nacional Autónoma de México. Oaxaca de Juárez, Oaxaca, México 68000}
\email{rita@im.unam.mx}
\email{porfirio.leon@im.unam.mx}

\author[Porfirio L. León Álvarez ]{Porfirio L. León Álvarez}



\date{}



\begin{abstract}
In this note we give an upper bound for the virtually cyclic dimension of any  normally poly-free group  in terms of its length. In particular, this implies that 
virtually even Artin groups of FC-type admit a finite dimensional model for the classifying space with respect to the family of virtually cyclic subgroups.
\end{abstract}
\maketitle


\section{Introduction}
\noindent Given a group $G$, we say that a collection $\calF$ of subgroups of $G$ is a \emph{family} if it is non-empty, closed under conjugation, and under taking subgroups. For a given a family $\calF$ of subgroups of $G$,  a $G$-CW-complex $X$ is a model for the classifying space $E_{\calF}G$  if  all of its isotropy groups belong to $\calF$  and the fixed point set $X^H$ is contractible whenever $H$ belongs to $\calF$.  It can be shown that a model for the classifying space   $E_{\calF}G$ always exists and it is unique up to $G$-homotopy equivalence.  In particular, the classifying spaces for the family $\fin$ of finite subgroups  of $G$ and the family $\vcyc$ of virtually cyclic subgroups of $G$, denoted by $\underline{E}G$ and $\underline{\underline{E}}G$ respectively,  are relevant due to their connection with the Farrell-Jones and Baum-Connes isomorphism conjectures; see for example \cite{LR05}.

The \emph{$\calF$-geometric dimension} of $G$  is defined as $$\gd_{\calF}(G)=\min\{n\in \mathbb{N}| \text{ there is a model for } E_{\calF}G \text{ of dimension } n \}.$$
For the trivial family and for the families $\fin$ and $\vcyc$, the number $\gd_{\calF}(G)$ is usually denoted by $\gd(G)$, $\gdfin(G)$ and $\gdvc(G)$, respectively. The $\calF$-geometric dimension has its algebraic counterpart, the \emph{$\calF$-cohomological dimension} $\cd_{\calF}(G)$, which can be defined in terms of Bredon cohomology. 
 The $\calF$-geometric dimension and the $\calF$-cohomological dimension satisfy the following inequality: 
\[\cd_{\calF}(G)\leq \gd_{\calF}(G)\leq \max\{\cd_{\calF}(G), 3\}.\]

In this note we study the geometric dimensions $\gd(G)$, $\gdfin(G)$ and $\gdvc(G)$ when $G$ is a normally poly-free group. A group $G$ is called  \emph{poly-free} if there exists a finite filtration  of $G$ by subgroups   $$1=G_0<G_1<\cdots< G_{n-1}< G_n=G$$  such that $G_i$ is normal in $G_{i+1}$, and the quotient $G_{i+1}/G_i$ is a free group, for $0\leq i\leq n-1$.  If  we have that each $G_i$ is normal in $G$, 
 we say that $G$ is \emph{normally poly-free}. If there is a filtration such that the free groups $G_{i+1}/G_i$ are of finite rank, we say that $G$ is \emph{poly-f.g.-free}. We define the \emph{length of $G$} as the minimum $n\in \mathbb{N}$ such that there is a filtration as before.  Poly-free groups are torsion-free, locally indicable, have finite asymptotic dimension  and satisfy the Baum–Connes Conjecture with coefficients \cite[Remark 2]{MR4246787}. Furthermore, it has been proved that normally poly-free groups satisfy the Farrell-Jones conjecture \cite[Theorem A]{MR4246787}, see also \cite{MR1797585},  \cite[Theorem 1.1]{MR4565703}, \cite[Theorem 2.3.7]{JPSS}. 
 
 In the literature, there are several examples of  poly-free and  normally poly-free groups. For instance, free groups and free by infinite cyclic groups are normally poly-free groups of length $\leq 2$.  Poly-$\mathbb{Z}$ groups are a particular case of  poly-free groups and their geometric and virtually cyclic dimensions have been completely characterized in \cite[Section 5]{LW12}. Furthermore,  pure braid groups of surfaces with nonempty boundary are known to be normally poly-free \cite{MR1797585} and so are even Artin groups of  $FC$-type  \cite[Theorem 3.18]{MR3900773}, \cite[Theorem A]{MR4360030}. It is an open question \cite[Question 2]{MR1714913} whether all Artin groups are virtually poly-free. 

We give  upper bounds for the geometric dimensions $\gd(G)$, $\gdfin(G)$ and $\gdvc(G)$ of any normally poly-free group $G$ in terms of its length.   

\begin{theorem}\label{cyclic:dimension:nor:poly:free:1} Let $G$ be a poly-free group of length $n\in\mathbb{N}$. 
\begin{enumerate}[a)]
    \item The geometric dimension $\gd(G)=\gdfin(G)$ is bounded above by $n$. Furthermore, if $G$ is normally poly-f.g.-free, then $\gd(G)= n$.
\item If  $G$ is  a  normally  poly-free group,  then the virtually cyclic dimension satisfies $$\gdvc(G)\leq 3(n-1)+2.$$ 
\end{enumerate}
 \end{theorem}

By \cite[Theorem 2.4]{MR1757730}, it follows from Theorem \ref{cyclic:dimension:nor:poly:free:1}(b) that any virtually normally poly-free group  admits a finite dimensional model  for the classifying space with respect to the family of virtually cyclic subgroups. Some examples include:

 \begin{itemize}
  \item  Virtually even  Artin groups of FC-type.  
 \item  The braid group $B_n(S)$ and the pure braid group $P_n(S)$ of $n$ strings on a connected compact surface $S$ with non-empty boundary.   
\end{itemize}

For Artin braid groups $B_n=B_n(\mathbb{D}^2)$ and pure braid groups $P_n=P_n(\mathbb{D}^2)$ the virtually cyclic geometric dimension  was explicitly computed in \cite{Ramon:Juan}. The existence of finite dimensional models for $\underline{\underline{E}}B_n(S)$ and  $\underline{\underline{E}}P_n(S)$ also follows from \cite[Theorem 1.4]{Nucinkis:Petrosyan} and the Birman exact sequence when the underlying surface $S$ is hyperbolic. However, the upper bounds that we get from \cref{cyclic:dimension:nor:poly:free:1}, when the surface $S$ has non-empty boundary, only depends of $n$ and not of the topology of the underlying surface $S$.  

\begin{remark} We don't expect  the upper bound obtained  in \cref{cyclic:dimension:nor:poly:free:1} (b) to be optimal.  For instance, for $n\ge2$,  if $G$ is a poly-$\mathbb{Z}$ group of length $n$, then $\gdvc(G)\leq n+1$, see \cite[Theorem 5.13]{LW12}. Furthermore, the pure braid group $P_n$ is normally poly-free of length $\leq n-1$ and $\gdvc(P_n)=n$ for $n\geq 3$, see \cite[Corollary 5.9]{Ramon:Juan}.  When the group $G$ is free-by-cyclic, we prove in \cref{free:by:cyclic:dimension} below that $\gdvc(G)\leq 3$.
\end{remark}

Our \cref{cyclic:dimension:nor:poly:free:1} is proved by an induction argument on the length of the normally poly-free group $G$. The proof of part (b) uses, as the base for the induction, that the virtually cyclic geometric dimension of a non-abelian free groups is equal to $2$.  This is known to hold for finitely generated non-abelian free groups \cite{MR2248975}, and we prove it for general non-abelian free groups in \cref{cyclic:dim:countable:free:group}. Our argument uses the next result that may be of independent interest.

\begin{theorem}\label{virt:dim:of:grps:gd:1}
Let $G$ be a group such that $\gdfin(G)=1$, then $\gdvc(G)\leq 2$.  
 Moreover, if $G$ is not virtually cyclic and has an element of infinite order,  then  $\gdvc(G)= 2$.
\end{theorem}

\begin{remark}
There are groups $G$ that are not virtually free and that satisfy $\gdfin(G)=1$, and hence $\gdvc(G)\leq 2$ by our \cref{virt:dim:of:grps:gd:1}.  If $G$ is the fundamental group of a graph of groups in which all vertex groups are finite and such that the orders of the vertex groups are not uniformly bounded, then $\gdfin(G)=1$, but the group $G$  is not virtually free; see for instance \cite[Theorem 7.3]{MR0564422} and references therein. An example of such a group is  $G=\Z_2*(\Z_2\times \Z_2)*(\Z_2\times \Z_2\times \Z_2)*\cdots$, which is not finitely generated.
In fact, since such $G$ is not virtually cyclic and has elements of infinite order, \cref{virt:dim:of:grps:gd:1} implies that $\gdvc(G)=2$. 
\end{remark}

\subsection*{Acknowledgements}
 We would like to thank 
Jesús Hernández Hernández for useful discussions. We are specially grateful to Luis Jorge Sánchez Saldaña for several comments on a draft of this paper and for suggesting the argument to prove \cref{lemma:condition C}. This paper was partially written while the first author was visiting Northeastern University with funding from the National University of Mexico through a DGAPA-UNAM PASPA sabbatical fellowship. She is grateful to DGAPA-UNAM and thanks the NEU Department of Mathematics and Solomon Jekel for their hospitality. The second author was supported by a doctoral scholarship of the Mexican Council of Humanities, Science and Technology (CONAHCyT).  We are grateful for the financial support of DGAPA-UNAM grant PAPIIT IA106923.

\section{Proof of Theorems \ref{cyclic:dimension:nor:poly:free:1} and \ref{virt:dim:of:grps:gd:1}}

We first prove \cref{cyclic:dimension:nor:poly:free:1}(a) about the geometric dimension of poly-free groups.  Notice that since poly-free groups are torsion free, then $\gdfin(G)=\gd(G)$.\\

\noindent{\bf \cref{cyclic:dimension:nor:poly:free:1}(a) }{\it If $G$ is a poly-free group of length $n$, then $\gd(G)\leq n$. Furthermore, if $G$ is normally poly-f.g.-free of lenght $n$ we have that $\gd(G)=n$.}
\begin{proof} The proof is by induction on the length $n$ of the poly-free group. If $n=1$, then $G$ is a non-trivial free group and   $\gd(G)= 1$. Suppose that the claim is true for poly-free groups of length $n\leq k-1$ 
and let $G$ be a poly-free group of length $k$. By definition, there is a filtration of $G$ by subgroups  $1=G_0\leq G_1\leq \cdots\leq G_{k-1}\leq G_k=G$   
 such that $G_{i+1}/G_i$ is a free group for $0\leq i\leq k-1$. Consider the following short exact sequence $$1\to G_{k-1}\to G\to G/G_{k-1}\to 1.$$
Notice that $G/G_{k-1}$ is a free group with $\gd(G/G_{k-1})=1$ and  $\gd(G_{k-1})=k-1$ by  induction hypothesis since $G_{k-1}$ is a poly-free group of length $\leq k-1$. Then, it follows from  \cite[Theorem 5.15]{Lu05}  that  $$\gd(G)\leq \gd(G_{k-1})+\gd(G/G_{k-1}) \leq (k-1) +1=k.$$ 
If $G$ is a normally  poly-f.g.-free of length $n$,  it follows from  \cite[Theorem 16]{MR596323}  that the homological dimension of $G$ over $\mathbb{Q}$ is given by $\hd_{\mathbb{Q}}(G)=n$. Since $\gd(G)\geq\cd(G)\geq \hd_{\mathbb{Q}}(G)$, the furthermore part of the statement follows.  
\end{proof}

 For $n=1,2$, the equality $\gd(G)=n$ holds for examples of normally poly-free groups $G$ of length $n$ that may not be finitely generated. That is the case when $G$ is a free group or a free-by-cyclic group; see \cref{free:by:cyclic:dimension} (a) below.  

The proof of \cref{cyclic:dimension:nor:poly:free:1} (b) is also done by induction on the length of the poly-free group.  We first need  some preparatory results. We prove in \cref{cyclic:dim:countable:free:group} that the virtually cyclic dimension of non-abelian free groups is equal to $2$ and show in \cref{free:by:cyclic:dimension} that it is at most $3$ for free-by-cyclic groups. 

In order to include non-finitely generated groups, we prove first \cref{virt:dim:of:grps:gd:1}. 
 Our argument was inspired by \cite[Section 6]{MR4472883}. It uses the fact that every virtually cyclic group acting on a simplicial tree $T$ fixes a vertex or acts co-compactly on a unique geodesic line; see \cite[Lemma 1.1]{DS99}. Recall that a geodesic line of a tree $T$ is a simplicial embedding of $\mathbb{R}$ in $T$, where  $\mathbb{R}$ has a vertex set $\mathbb{Z}$ and an edge joining any two consecutive integers.\medskip

\noindent{\bf 
\cref{virt:dim:of:grps:gd:1}}
Let $G$ be a group such that $\gdfin(G)=1$, then $\gdvc(G)\leq 2$.  
 Moreover, if $G$ is not virtually cyclic and has an element of infinite order,  then  $\gdvc(G)= 2$.
\begin{proof}
Let $G$ be a group with $\gdfin(G)=1$. Then there is a simplicial tree $T$ which is a model for the classifying space $\underline{E}G$. We promote $T$ to a model for $\underline{\underline{E}}G$ by coning-off on $T$ some geodesics as we now explain.

First we prove that the set-wise stabilizer $\stab_{G}(\gamma)$ of any geodesic line $\gamma$ in $T$ is a virtually cyclic group. Consider $\gamma$ with the simplicial structure induced by $T$. Note that the group $\Aut(\gamma)$ of simplicial automorphisms of $\gamma$ is isomorphic to the infinite dihedral group $D_{\infty}$. Since $\stab_{G}(\gamma)$ acts  by simplicial automorphisms on  $\gamma$, then  there is a  homomorphism of groups $\varphi \colon \stab_{G}(\gamma) \to \Aut(\gamma)=D_{\infty}$. Let us denote by $D$ the image of $\varphi$ and notice that it is a virtually cyclic group since it is  subgroup of a $D_{\infty}$. On the other hand,  $\ker(\varphi)$ fixes point-wise the vertices of $\gamma$. Since $T$ is a model for $\underline{E}G$, we have that $\ker(\varphi)$ must be a finite group.  It  follows  from  the short exact sequence
$$1\to \ker(\varphi) \to \stab_{G}(\gamma)\xrightarrow[]{\varphi} D\to 1 $$ that $\stab_{G}(\gamma)$ is a virtually cyclic group.

Since $T$ is a model for $\underline{E}G$, any infinite virtually cyclic subgroup of $G$ must act co-compactly in a unique geodesic line of $T$.  
Let  $\mathcal{A}$ be the collection of all the geodesics of  $T$  that admit a  co-compact action of an infinite virtually cyclic subgroup of $G$. Consider the space $\Hat{T}$ given by the following  homotopy $G$-push-out:
\begin{equation*}\label{model:free:group:1}
 \xymatrix{\displaystyle\bigsqcup_{\gamma\in \mathcal{A}}\gamma\ar[d] \ar[r]  & T\ar[d]\\
\displaystyle\bigsqcup_{\gamma\in \mathcal{A}} \{*_{\gamma}\}  \ar[r] & \Hat{T} 
}
\end{equation*}
If $H\leq G$ acts co-compactly on the geodesic line $\gamma$ of $T$ and $g\in G$, then $gHg^{-1}$ acts co-compactly on $g\gamma$. It follows that both $\bigsqcup_{\gamma\in \mathcal{A}}\gamma$ and $\bigsqcup_{\gamma\in \mathcal{A}}\{*_\gamma\}$ are $G$-CW-complexes, and therefore the space  $\Hat{T}$ is a   $G$-CW-complex of dimension 2. 

We claim that  $\Hat{T}$ is a model for $\underline{\underline{E}}G$.  To show this  we need to check the following:
\begin{enumerate}[a)]
    \item For all $x\in \Hat{T}$ the isotropy group $\stab_G(x) \in \vcyc$.
    \item The fixed point set $\Hat{T}^{H}=\{x\in \Hat{T} | \ hx=x, \text{for all }h\in H\}$ is contractible if $H\in\vcyc$.
\end{enumerate}

Item a) follows from the construction of  $\Hat{T}$. Indeed, we have two cases $x\in T$ or $x\in \Hat{T}-T$. Observe that in the first case we have that the isotropy group $\stab_G(x)$ is finite. In the second case, if $x\in \Hat{T}-T$ is a conic point, then \(\stab_G(x)\) is infinite virtually cyclic; otherwise the isotropy \(\stab_G(x)\) is contained in the stabilizer of a conic point, hence it is virtually cyclic.

It remains to show that item b) holds. 
Let $H\in \vcyc$  and consider the action of $H$ on the tree $T$ obtained by restricting the action of $G$ on $T$. It follows that $H$ fixes a vertex of $T$ or  it acts co-compactly on a unique geodesic line $\gamma$ of $T$. 

If $H$ fixes a vertex of $T$, then $H$ is a finite group since $T$ is a model for $\underline{E}G$. This implies that $T^H$ is a non-empty subtree of $T$. Therefore  $\hat{T}^{H}$ is obtained from  $T^H$ possibly coning-off some geodesics segments, we conclude that $\hat{T}^{H}$  is contractible. 

Otherwise, we have that $H$ acts co-compactly in a unique geodesic line $\gamma$ on $T$, then $\gamma\in \mathcal{A}$ and $*_{\gamma}\in \Hat{T}^H$. Notice that for any $x\in \Hat{T}-\bigsqcup_{\gamma\in \mathcal{A}}\{*_\gamma\}$ the isotropy group $\stab_G(x)$ is finite.  Indeed, if $x\in T$ this claim follows from the fact that $T$ is a model for $\underline{E}G$. Now if $x\notin T$, then  there  is some  $\gamma \in \mathcal{A}$ such that  $x$ lies in the interior of the segment from a point in $\gamma$ to the coin point $*_\gamma$. In particular, $\stab_G(x)$ is contained in the stabilizer of this segment given by the intersection of the isotropy groups of the end points, one of which is a finite group. Then  $\Hat{T}^{H}\subseteq \bigsqcup_{\gamma\in \mathcal{A}} \{*_\gamma\}$ and by the uniqueness of the geodesic line $\gamma$, we conclude that $\Hat{T}^{H}= \{*_\gamma\}$. 

We now prove the moreover part of the theorem.  Assume that $G$ is not virtually cyclic and let $h\in G$ be an element of infinite order. The cyclic subgroup of $G$ generated by $h$ must act co-compactly in a unique geodesic line $\gamma$ of $T$. Since the stabilizer of any geodesic line in $T$ is virtually cyclic and the group $G$ is not, there exists an element $g\in G$ that is not in $\stab_{G}(\gamma)$. Notice that the subgroup $H$ generated by $h$ and $g$ 
cannot be virtually cyclic. Indeed, if $H$ is virtually cyclic, then it must act co-compactly in a unique geodesic line $\beta$ of $T$, and so do its cyclic subgroups generated by $h$ and $g$.  By uniqueness of the the geodesic $\gamma$ stabilized by $h$, we must have that $\beta=\gamma$. Then $H$ stabilizes $\gamma$ and, in particular,  $g\gamma=\gamma$ which contradicts the fact that we are taking $g\notin \stab_{G}(\gamma)$.  By \cite[Lemma 2.2]{MR4472883} we have that $2\le \gdvc(H) \le \gdvc(G)$.
\end{proof}

The following result was proved in \cite{MR2248975}  for finitely generated virtually free groups that are not virtually cyclic. 
\begin{corollary}\label{cyclic:dim:countable:free:group}
Let $G$ be a virtually free group which is not virtually cyclic. Then $\gdvc(G)=2$.
\end{corollary}
\begin{proof}
Let $G$ be a non-trivial virtually free group. Then  
$G$ is the fundamental group of a graph of groups in which all vertex groups are finite; see for example \cite[Theorem 7.3]{MR0564422}  and notice that this structure result holds for virtually free groups that may not be finitely generated. 
 This implies that
 $\gdfin(G)=1$. Since $G$ is not virtually cyclic  it follows from  \cref{virt:dim:of:grps:gd:1} that $\gdvc(G)= 2$.  
\end{proof}

The induction step in the proof to of \cref{cyclic:dimension:nor:poly:free:1} (b) requires having an upper bound for the virtually cyclic dimension of a free-by-cyclic group. To obtain it in \cref{free:by:cyclic:dimension} below, we use the following  condition  introduced by Lück.

\begin{definition}\cite[Condition 4.1]{MR2545612}
 We say that a group $G$ satisfies {\it condition $(C)$} if for every $g,h\in G$ with $|h|=\infty$ and $k,l\in \Z$ we have that $gh^kg^{-1}=h^l$ implies that $|k|=|l|$. 
\end{definition}

 \begin{lemma}\label{lemma:condition C}
 Consider a short exact sequence of groups $$1\to F\to G \to \Z \to 1$$  such that $F$ is a  free group. Then the group $G$ satisfies condition $(C)$.
 \end{lemma}
 \begin{proof}
 Consider the automorphism $\varphi:F\rightarrow F$ such that $G\cong F\rtimes_{\varphi} \Z$.
 
Let $(x,a),(y,b)\in F\rtimes_{\varphi} \Z$ with $|(x,a)|=\infty$ and $k,l\in \Z$ such that $$(y,b)(x,a)^{k}(y,b)^{-1}=(x,a)^l.$$ Then $b+ka-b=la$; if $a\neq 0$, it follows that $k=l$.

Now assume that $a=0$ and $x\neq e_F$. Then we have  that $(y,b)(x^k,0)(y,b)^{-1}=(x^l,0)$, which implies that $y\varphi^b(x^k)y^{-1}=x^l$ in $F$. In other words, 
\begin{equation}\label{eq:subgroup} c_y\circ\varphi^b(x^k)=x^l,\end{equation} where $c_y:F\rightarrow F$ is the automorphism of $F$ given by conjugation by $y$. 

It is well known that in a free group, every  infinite cyclic subgroup \( C \) is contained in a unique maximal cyclic subgroup \( C_{\text{max}} \).
 Consider the infinite cyclic subgroup $C=\langle x\rangle$ of $F$.
From \cref{eq:subgroup}, we have that the automorphism $c_y\circ\varphi^b$ of $F$ sends the subgroup $\langle x^k\rangle$ of $C$ into the subgroup $\langle x^l\rangle$ of $C$, then it must take $C_{max}$ isomorphically onto $C_{max}$. Therefore,  $c_y\circ\varphi^b|_{C_{max}}$ is an automorphism of the infinite cyclic group $C_{max}$ and hence $c_y\circ\varphi^b|_{C_{max}}=\pm \text{Id}_{C_{max}}$. It follows that  $x^l=c_y\circ\varphi^b(x^k)=\pm \text{Id}_{C_{max}}(x^k)=x^{\pm k}$, then $|k|=|l|$. 
 \end{proof}

 \begin{remark} More generally, notice that the argument in \cref{lemma:condition C} shows that a semi-direct product $G\cong F\rtimes_{\varphi} \Z$ satisfies condition $(C)$ whenever the group $F$ is torsion free and satisfies that
 ``any infinite  cyclic subgroup $C$ of $F$ is contained in a unique maximal cyclic subgroup $C_{max}$ of $F$''. See \cite[Section 3]{LW12} for examples of such groups.
 \end{remark}
 
 Given $H$  a subgroup of $G$, we denote the  normalizer of $H$ in $G$ by $N_G(H)$ and the corresponding Weyl group by  $W_G(H)=N_G(H)/H$.
 
\begin{lemma}\cite[Lemma 4.4]{MR2545612}\label{luck:condition C}
 Let $n$ be an integer. Suppose  that $G$ satisfies condition $(C)$. Suppose that $\gdfin(G)\leq n$ and for every infinite cyclic subgroup $H$ of $G$ we have $\gdfin(W_G(H))\leq n$. Then $\gdvc(G)\leq n+1$.  
\end{lemma}

\begin{proposition}\label{free:by:cyclic:dimension}
 Consider a short exact sequence of groups $$1\to F\to G \to \Z \to 1$$  such that $F$ is a  free group. 
 \begin{enumerate}[a)]
    \item  The geometric dimension $\gd(G)\leq 2$, and equality holds if the group $G$ is not free. 
    \item The virtually cyclic geometric dimension satisfies ${2\leq}\gdvc(G)\leq 3$. 
 \end{enumerate}
\end{proposition}

\begin{proof} If $G$ is a free group, then $\gd(G)=1$ and $\gdvc(G)\leq 2$ by \cref{cyclic:dim:countable:free:group}. 

Suppose that $G$ is not a free group.  First we prove item $a)$. Since $G$ is a poly-free group of length $2$, by \cref{cyclic:dimension:nor:poly:free:1}(a) we have  that $\gd(G)\leq 2$. On the other hand, since $G$ is non-trivial and non-free, then  $\cd(G)\geq 2$ by the Stallings and Swan theorem. Hence $\cd(G)=\gd(G)= 2$.

We prove item $b)$.  For the lower bound notice that if $F=\mathbb{Z}$, then $G$ is isomorphic to $\Z^2$ or $\Z\rtimes \Z$, in both cases $G$ is a 2-crystallographic group, it follows from \cite{CFH06} that $\gdvc(G)=3$. If $F$ is not cyclic, then it follows from \cref{cyclic:dim:countable:free:group} that $2=\gdvc(F)\leq \gdvc(G)$. 

 We prove next that $\gdvc(G)\leq3$.
By \cref{lemma:condition C}, a free-by-cyclic group $G\cong F\rtimes \Z$ satisfies condition $(C)$. Therefore, from \cref{luck:condition C} and item $a)$, it is enough to show that the geometric dimension $\gdfin(W_G(H))\leq 2$ for any infinite cyclic subgroup $H$ of $G$. 

Let $H$ be a cyclic subgroup of $G$. From the short exact sequence $1\to F\to G\xrightarrow[]{p} \Z\to 1 $ we have, by restriction, the  short exact sequence 
 \begin{equation}\label{ses:1}
     1\to F\cap N_G(H)\to  N_G(H)\to Q\to 1,
 \end{equation}
 where $Q$ is a subgroup of $\Z$.  We need to consider two cases: $i)$ $p(H)=0$ or  $ii)$ $p(H)\neq 0$. 
 
In case $i)$,  we assume that $p(H)=0$, then  $H\subseteq F$. Since $F$ is a free group, it follows that   $F\cap N_G(H)= N_F(H) $ is a cyclic subgroup of $F$  and $(F\cap N_G(H))/H$ is finite. From the short exact sequence (\ref{ses:1})  we obtain   
$$1\to (F\cap N_G(H))/H\to  W_G(H)\to Q\to 1,$$
and therefore $W_G(H)$ is a virtually cyclic group.

In the case $ii)$, we have $p(H)\neq 0$ and $Q/p(H)$ is a finite cyclic group. From the short exact sequence (\ref{ses:1}) we get
$$1\to F\cap N_G(H)\to  W_G(H)\to Q/p(H)\to 1,$$
where $F\cap N_G(H)$ is a free group. Hence $W_G(H)$ is a virtually free group. 

In both cases $i)$ and $ii)$ we have that $\gdfin(W_G(H))\leq 1$ and hence $\gdvc(G)\leq3$.
\end{proof}

\begin{remark}
 If we assume in \cref{free:by:cyclic:dimension}  that $G$ is a countable group, then the result can be deduced from \cite[Corollary 3]{MR3570151}.  See also \cite[Theorem B]{MR3252961} for a criterion for more general groups that fit into an extension with torsion-free quotient to admit a finite-dimensional classifying space with virtually cyclic stabilizers. \end{remark}

The following result of Lück and Weiermann allows us to relate the geometric dimensions associated to nested families of subgroups.

\begin{lemma}\cite[Proposition 5.1 (i)]{LW12}\label{lw:nested families}
    Let $\mathcal{F}\subseteq \mathcal{G}$ be two families of subgroups of a group $G$. Let $n\geq 0$ be an integer such that for any $H\in\mathcal{G}$, there is an $n$-dimensional model for $E_{\mathcal{F}\cap H}(H)$. Then $\gd_{\mathcal{F}}(G)\leq \gd_{\mathcal{G}}(G)+n$. 
\end{lemma}

We are ready to prove \cref{cyclic:dimension:nor:poly:free:1} (b).\medskip

\noindent{\bf \cref{cyclic:dimension:nor:poly:free:1} (b) }  Let  $G$ be  a  normally  poly-free group of length $n$.  Then $\gdvc(G)\leq 3(n-1)+2$.

\begin{proof} The proof is done by induction on $n$.
If $n=1$, the statement follows from \cref{cyclic:dim:countable:free:group}. Suppose that the result holds for normally poly-free groups of length $n\leq k-1$.

Let $G$ be a  normally  poly-free group  of length $k$. Then there is a filtration  of $G$ by subgroups   $1=G_0<G_1<\cdots< G_{k-1}< G_k=G$  such that $G_i$ is normal in $G$, and the quotient $G_{i+1}/G_i$ is a  free group. We consider the following short exact sequence 
$$1\to G_{1}\to G\xrightarrow[]{p} G/G_{1}\to 1.$$ 
Note that $G/G_1$ is a  normally  poly-free group of length $\leq k-1$. Consider the family pull-back $p^{*}(\calG)$ of the family $\calG$  of virtually cyclic subgroups of $G/G_1$, i.e. $p^{*}(\calG)$ is the family of subgroups of $G$ generated by   $$\{p^{-1}(L): L \text{ is a virtually cyclic subgroup of } G/G_1 \}.$$
Note that a model $X$ of $E_{\calG}(G/G_1)$ is a model of $E_{p^{*}(\calG)}G$ via the action given by projection  $p$, then  $\gd_{p^{*}(\calG)}(G)\leq \gd_{\calG}(G/G_1)=\gdvc(G/G_1)$. Let $\vcyc$ denote the family of virtually cyclic subgroups of $G$ and notice that $\vcyc\subseteq p^{*}(\calG)$. By \cref{lw:nested families} we have

\begin{equation}\label{eq1}
\begin{split}
    \gdvc(G)&\leq \gdvc (G/G_1)+ \max\{\gd_{\vcyc\cap p^{-1}(L)}(p^{-1}(L)): \ L\in\mathcal{G}\}. 
    \end{split}
\end{equation}
We show that  $\gd_{\vcyc\cap p^{-1}(L)}(p^{-1}(L))\leq 3$ for any $L\in\mathcal{G}$. 

If $L$ is the trivial group, then $p^{-1}(L)\cong G_1$ is a free group and $\gd_{\vcyc\cap p^{-1}(L)}(p^{-1}(L))=\gdvc(p^{-1}(L))= 2$  by \cref{cyclic:dim:countable:free:group}. Otherwise, $L$ is an infinite cyclic subgroup of the torsion free group $G/G_1$. From the short exact sequence $$1\to G_1 \to  p^{-1}(L)\to L\to 1,$$ we see that $ p^{-1}(L)$ is a free-by-cyclic group. By \cref{free:by:cyclic:dimension} we have that $$\gd_{\vcyc\cap p^{-1}(L)}(p^{-1}(L))=\gdvc(p^{-1}(L))\leq 3.$$

From \cref{eq1} and the induction hypothesis we have $\gdvc(G)\leq \gdvc(G/G_1)+ 3\leq3(k-1)+2.$
\end{proof}

\bibliographystyle{alpha} 
\bibliography{mybib}
\end{document}

%% file: Preamble.tex
\usepackage{graphicx}
\usepackage{pinlabel} 
\usepackage{amsmath}
\usepackage{amsfonts}
\usepackage{amssymb}
\usepackage{mathtools}
\usepackage[all,cmtip]{xy}
\usepackage{amsthm}
\usepackage{tikz-cd}
\usepackage{comment}
\usepackage{enumerate}
\usepackage{url}
\usepackage{epsfig}
\usepackage{hyperref}
\usepackage[utf8]{inputenc}
\usepackage[capitalise]{cleveref}
\usepackage{geometry}
\makeatletter
\providecommand\@dotsep{5}
\def\listtodoname{List of Todos}
\def\listoftodos{\@starttoc{tdo}\listtodoname}
\makeatother

\newtheorem{theorem}{Theorem}[section]
\newtheorem{proposition}[theorem]{Proposition}
\newtheorem{corollary}[theorem]{Corollary}
\newtheorem{lemma}[theorem]{Lemma}

\newcommand{\mycomment}[1]{}

  \theoremstyle{definition}
\newtheorem{definition}[theorem]{Definition}

\newtheorem{remark}{Remark}

\newcommand{\Z}{\mathbb{Z}}

\newcommand{\calF}{{\mathcal F}}

\newcommand{\calG}{{\mathcal G}}

\newcommand{\vcyc}{V\text{\tiny{\textit{CYC}}}}

\newcommand{\fin}{F\text{\tiny{\textit{IN}}}}

\newcommand{\nbeq}{\begin{equation}}
\newcommand{\neeq}{\end{equation}}
\newcommand{\beq}{\begin{equation*}}
\newcommand{\eeq}{\end{equation*}}

\DeclareMathOperator{\hd}{hd}
\DeclareMathOperator{\cd}{cd}

\DeclareMathOperator{\gd}{gd}
\DeclareMathOperator{\gdfin}{\underline{gd}}
\DeclareMathOperator{\gdvc}{\underline{\underline{gd}}}

\DeclareMathOperator{\Aut}{Aut}

\DeclareMathOperator{\stab}{Stab}

